\documentclass[11pt]{amsart}

\usepackage[margin=1.5in]{geometry} 
\usepackage[utf8]{inputenc}
\usepackage[T1]{fontenc}
\usepackage{amsmath,amsthm,amssymb}
\usepackage{thmtools}
\usepackage{hyperref}
\hypersetup{colorlinks=false, linkcolor=blue}
\usepackage{cleveref}
\usepackage{dsfont}
\usepackage{xcolor}
\usepackage{calligra,mathrsfs}
\usepackage[shortlabels]{enumitem}
\usepackage{array}
\usepackage{setspace}
\usepackage[sortcites=true, style=alphabetic]{biblatex}
\AtEveryBibitem{\clearfield{url}}

\newcommand{\af}{\mathfrak a}
\newcommand{\bfr}{\mathfrak b}

\newcommand{\ess}{\text{ess}}
\newcommand{\pf}{\mathfrak p}
\newcommand{\qf}{\mathfrak q}
\newcommand{\mf}{\mathfrak m}

\newcommand{\A}{\mathbb A}
\newcommand{\F}{\mathbb F}
\newcommand{\R}{\mathbb R}
\newcommand{\Z}{\mathbb Z}
\renewcommand{\P}{\mathbb P}

\newcommand{\Lc}{\mathcal L}
\newcommand{\Mmc}{\mathcal M}
\newcommand{\Pc}{\mathcal P}

\renewcommand{\a}{\alpha}
\renewcommand{\b}{\beta}
\renewcommand{\l}{\lambda}

\DeclareMathOperator{\fpt}{fpt}
\DeclareMathOperator{\lct}{lct}
\DeclareMathOperator{\Hom}{Hom}
\DeclareMathOperator{\Spec}{Spec}
\DeclareMathOperator{\ini}{in}
\DeclareMathOperator{\vol}{vol}
\DeclareMathOperator{\conv}{conv}
\DeclareMathOperator{\hgt}{height}

\newcommand{\del}{\partial} 

\DeclareMathOperator{\codim}{codim}
\DeclareMathOperator{\PNV}{(\mathbb{P}^n)^\vee}

\newcommand{\ceil}[1]{\lceil #1\rceil}
\newcommand{\floor}[1]{\left\lfloor#1\right\rfloor}

\renewcommand{\char}{\textup{char }}
\renewcommand{\bar}{\overline}

\theoremstyle{plain}
\newtheorem{Theoremx}{Theorem}
 
\newtheorem{theorem}{Theorem}[section]
\newtheorem{prop}[theorem]{Proposition}
\newtheorem{lemma}[theorem]{Lemma}

\newtheorem*{cor*}{Corollary}

\newtheorem*{theorem*}{Theorem}

\theoremstyle{definition}
\newtheorem{defn}[theorem]{Definition}

\newtheorem{example}[theorem]{Example}
\newtheorem{remark}[theorem]{Remark}
\newtheorem{convention}[theorem]{Convention}

\crefname{prop}{proposition}{propositions}
\Crefname{prop}{Proposition}{Propositions}
\crefname{defn}{definition}{definitions}
\Crefname{defn}{Definition}{Definitions}

\usepackage{etoolbox}

\newtoggle{verbose}
\togglefalse{verbose}

\newcommand{\details}[1]{%
  \iftoggle{verbose}{%
    \par\noindent\textcolor{blue}{#1}\par%
  }{}%
}

\begin{document}

\title{On Lower Bounds for the $F$-Pure Thresholds of Equigenerated Ideals}
\author{Benjamin Baily}
\thanks{The author was supported by NSF grant DMS-2101075, NSF RTG grant DMS-1840234, and a Simons dissertation fellowship.}
\begin{abstract}
    Let $k$ be a field of positive characteristic and $R = k[x_0,\dots, x_n]$. We consider ideals $I\subseteq R$ generated by homogeneous polynomials of degree $d$. Takagi and Watanabe proved that $\fpt(I)\geq \hgt(I)/d$; we classify ideals $I$ for which equality is attained. Inspired by a result of de Fernex, Ein, and Musta\c{t}\u{a}, we give a lower bound on $\fpt(I)$ in terms of the height of $\tau(I^{\fpt(I)})$.
    \end{abstract}
\maketitle
\section{Introduction}
The $F$-pure threshold, introduced by Takagi and Watanabe \cite{takagi_f-pure_2004}, is a numerical singularity invariant of pairs in positive characteristic. The $F$-pure threshold was proposed as a positive characteristic analog of the log canonical threshold; whereas the log canonical threshold is widely studied in birational and complex geometry, the $F$-pure threshold better reflects the subtleties of singularities in prime characteristic.

We consider a pair $(R,I)$, where $R$ is a polynomial ring over a field and $I$ is generated by forms of degree $d$. In this setting, Takagi and Watanabe proved the following sharp lower bound on the $F$-pure threshold $\fpt(I)$:
\begin{prop}[\cite{takagi_f-pure_2004}, Proposition 4.2]\label{prop:tw-fpt-equigen}
    Let $k$ be a field of positive characteristic and set $R = k[x_0,\dots, x_n]$. Suppose $I\subseteq R$ is generated by homogeneous polynomials of degree $d$. If $h$ is the height of $I$, then $\fpt(I)\geq h/d$.
\end{prop}
If we instead consider a field of characteristic $0$ and the log canonical threshold (lct), much more is known. We refer the reader to \cite{pragacz_impanga_2012} for background on log canonical singularities and the lct.
\begin{theorem}[\cite{de_fernex_bounds_2003}, Theorems 3.4 and 3.5]\label{thm:equigen-lct}
    Let $k$ be an algebraically closed field of characteristic zero and set $R = k[x_0,\dots, x_n]$. Suppose $I = (f_1,\dots, f_r)\subseteq R$ is generated by forms of degree $d$. Let $e$ denote the codimension of $Z$, where $Z$ is the non-klt locus of $(R, I^{\lct(I)})$. Then we have $\lct(I) \geq e/d$ with equality if and only if there exist linear forms $\ell_1,\dots, \ell_e\in R$ such that $Z = V(\ell_1,\dots, \ell_e)$ and $f_i\in k[\ell_1,\dots, \ell_e]$ for all $1\leq i\leq r$. 
\end{theorem}
Our goal is to bridge the gap between \Cref{prop:tw-fpt-equigen} and \Cref{thm:equigen-lct}. As we show in \Cref{ex:naive-analog-failure}, a naive translation of \Cref{thm:equigen-lct} into characteristic $p$ is not true without an additional hypothesis. Towards the goal of bridging this gap, we contribute two results. The first is a classification of ideals for which the lower bound in \Cref{prop:tw-fpt-equigen} is sharp.
\begin{Theoremx}\label{thm:equigen-fpt}
    Let $k$ be an algebraically closed field of positive characteristic. Let $I$ be a homogeneous ideal in $k[x_0,\dots, x_n]$ generated by $d$-forms. If $h$ is the height of $I$, then $\fpt(I) = h/d$ if and only if the integral closure $\overline{I}$ of $I$ satisfies $\overline{I} = (x_0,\dots, x_{h-1})^d$ up to a change of coordinates.
\end{Theoremx}
The proof of \Cref{thm:equigen-fpt} goes as follows. First, we prove the claim in the case that $I$ is a complete intersection of height $n$ (\Cref{lem:min-fpt-char-CI-height-n}). In this case, let $\pf$ be a minimal prime over $I$. Since $\pf$ is the ideal of a point in $\P^n$, we may change coordinates so that $\pf = (x_1,\dots, x_n)$. We then transform \Cref{lem:min-fpt-char-CI-height-n} to a statement about the graded system of monomial ideals $\{\ini_{>_{\text{lex}}}(I^m)\}_{m>0}$, which we solve using convex geometry: after applying estimates for the Hilbert series of powers of $I$ (\Cref{lem:hilb-fn-power-of-CI}), the result is a consequence of a 1960 result of Gr\"unbaum (\Cref{thm:grunbaum-ineq}).

To generalize beyond the case of a complete intersection, we note that with $h := \hgt I$, then any $h$ general $d$-forms in $I$ generate a complete intersection $J\subseteq I$, and we show that $J$ is a reduction of $I$. To generalize beyond the case that $\hgt I = n$, we consider $I|_H$, where $H$ is a general hyperplane through the origin. By \Cref{prop:tw-fpt-equigen}, we have $h/d\leq \fpt(I|_H) \leq \fpt(I) = h/d$, so $\overline{I|_H} = (x_0,\dots, x_{h-1})^d$. By \Cref{prop:ess-dim-general-linear-space}, we deduce that $I$ has the same form.

\details{The application which motivated this article is the following, which we state in a special case for brevity. \begin{theorem}[\cite{baily_homogeneous_2026}, Theorem 5.11] Let $k = k[x_0,\dots, x_n]$ be a polynomial ring over an algebraically closed field of characteristic $p > 0$. Suppose $f_1,\dots, f_r$ is a homogeneous regular sequence such that $\deg(f_i) = d_i$. Set $I = (f_1,\dots, f_r)$. Then $\fpt(I) \geq \frac{1}{d_1} + \dots + \frac{1}{d_r}$ with equality if and only if $\overline{I} = \overline{(x_1^{d_1},\dots, x_r^{d_r})}$ in suitable coordinates.\end{theorem}
The proof of Theorem 5.11 in op. cit. is inductive on the number of distinct degrees $\{d_i\}$, and the base case is equivalent to Theorem \ref{thm:equigen-fpt}.}

Our second contribution is a version of \cite[Theorem 3.4]{de_fernex_bounds_2003} when the pair $(R, I)$ has well-behaved singularities at the $F$-pure threshold.
\begin{Theoremx}\label{thm:test-ideal-at-threshold}
    Suppose $\char k = p>0$ and $R = k[x_0,\dots, x_n]$. Let $I\subseteq R$ be an ideal generated by $d$-forms and let $c = \fpt(I)$. Let $h$ denote the height of the test ideal $\tau(R, I^c)$. If $(R, I^c)$ is sharply $F$-pure, then $c\geq \frac{h}{d}$.
\end{Theoremx}
The assumption that $(R, I^c)$ is sharply $F$-pure is necessary by \Cref{ex:naive-analog-failure}. When $c = h/d$, we show that $\tau(I^c)$ is generated by linear forms $\ell_1,\dots, \ell_h$ (\Cref{prop:test-ideal-in-equality-case}), but unlike the characteristic zero case (\Cref{thm:equigen-lct}), it may not be the case that $I$ is extended from $k[\ell_1,\dots, \ell_h]$ (\Cref{ex:CHSW}).
\section{Preliminaries}
\subsection{The \texorpdfstring{$F$}{F}-Pure Threshold}
For detailed background on the $F$-pure threshold, we direct the reader to \cite{schwede_sharp_2008,takagi_f-pure_2004}. We summarize key definitions and results.
\begin{defn}
    Let $R$ be a ring of characteristic $p>0$. We let $F_*R$ denote the $R$-module structure on $R$ given by restriction of scalars along the Frobenius map $F:R\to R$. We say $R$ is $F$-finite if $F_*R$ is module-finite over $R$. 
\end{defn}
\begin{defn}[\cite{schwede_sharp_2008}]
    Let $R$ be an $F$-finite ring, $I\subseteq R$ an ideal, and $t\in\R^+$. The pair $(R, I^t)$ is \textit{sharply $F$-split} if for some (equivalently, infinitely many) $e>0$, the map
    \[
    I^{\ceil{t(p^e-1)}}\cdot \Hom(F^e_*R, R)\to R
    \]
    is surjective. 
\end{defn}
\begin{defn}[\cite{takagi_f-pure_2004}]
    The \textit{$F$-pure threshold} of the pair $(R,I)$ is the supremum of all $t$ such that $(R, I^t)$ is sharply $F$-split. We denote this quantity by $\fpt(R, I)$, or $\fpt(I)$ when the ambient ring is clear. 
\end{defn}
In practice, the following proposition is a more useful characterization of the $F$-pure threshold.
\begin{prop}
    Let $k$ be an $F$-finite field of characteristic $p > 0$. Let $R$ be a polynomial ring over $k$, $\mf$ the homogeneous maximal ideal of $R$, and $I\subseteq R$ a homogeneous ideal. Then the $F$-pure threshold of the pair $(R, I)$ is equal to
    \[
    \sup\left\{\frac{\nu}{p^e}: I^\nu\notin \mf^{[p^e]}\right\}.
    \]
    In fact, let $\nu_I(p^e) = \max\{r: I^r\notin \mf^{[p^e]}\}$. Then the $F$-pure threshold of $(R, I)$ is equal to the limit $\lim_{e\to\infty}\nu_I(p^e)/p^e$. 
\end{prop}
\begin{proof}
    See \cite[Proposition 3.10]{de-stefani_graded_2018}.
\end{proof}
\begin{prop}[Properties of the $F$-pure threshold]\label{prop:fpt-properties}
    Let $R$ be an $F$-finite $F$-pure ring of characteristic $p>0$. Then for all ideals $I\subseteq R$ such that $I$ contains a nonzerodivisor, we have
    \begin{enumerate}[(i)]
        \item If $I\subseteq J$, then $\fpt(I)\leq \fpt(J)$.
        \item For all $m > 0$, we have $\fpt(I^m) = m^{-1}\fpt(I)$.
        \item We have $\fpt(I) = \fpt(\bar{I})$, where $\bar{I}$ denotes the integral closure of $I$. 
    \end{enumerate}
\end{prop}
\begin{proof}
    See \cite[Proposition 2.2]{takagi_f-pure_2004} (1), (2), (6).
\end{proof}
\begin{prop}\label{prop:semicont}
    Let $R = k[x_0,\dots, x_n]$. Let $>$ be a monomial order. Let $I\subseteq R$ be an ideal, and $\ini_>(I)$ the initial ideal of $I$ with respect to $>$. Then $\fpt(\ini_>(I))\leq \fpt(I).$
\end{prop}
\begin{proof}
    See the claim preceding \cite[Remark 4.6]{takagi_f-pure_2004}.
\end{proof}
\subsection{Newton Polytopes of Monomial Ideals}
When working with monomial ideals, one often identifies a monomial $x_0^{a_0}\cdots x_n^{a_n}$ with the point $(a_0,\dots, a_n)\in \Z_{\geq 0}^{n+1}$. For future reference, it will help to give a name to this identification.
\begin{defn}
    Let $k$ be a field. We define the map 
    \[\log: \{\text{monomials in }k[x_0,\dots, x_n]\}\to \Z_{\geq 0}^{n+1},\qquad\log(x_0^{a_0}\cdots x_n^{a_n}) = (a_0,\dots, a_n).\]
\end{defn}
\begin{defn}
    Let $\af\subseteq k[x_0,\dots, x_n]$ be a monomial ideal. Then the \textit{Newton Polytope} of $\af$, denoted $\Gamma(\af)$, is the convex hull in $\R^{n+1}$ of $\log(\af)$. Later on, we will let $\conv(-)$ denote the convex hull of a set.
\end{defn}
\begin{remark}
    We record several properties of $\Gamma(\af)$.
    \begin{enumerate}[(i)]
        \item $\Gamma(\af)$ is a closed, convex, unbounded subset of the first orthant of $\R^{n+1}$.
        \item When $\af$ is an $\mf$-primary ideal, the complement of $\Gamma(\af)$ inside the first orthant is an open, bounded polyhedron.
        \item For two ideals $\af, \bfr$, the Minkowski sum of $\Gamma(\af)$ and $\Gamma(\bfr)$ is equal to $\Gamma(\af\bfr)$. In particular, $\Gamma(\af^n) = n\Gamma(\af)$.
    \end{enumerate}
\end{remark}
For the proof of \Cref{thm:equigen-fpt}, we will also require the following.
\begin{defn}\label{defn:simplex}
    We define the standard $n$-simplex $\Delta_n\subseteq \R^{n+1}$ as follows:
    \[
    \Delta_n = \{(a_0,\dots, a_n): 0\leq a_i, a_0 +\dots + a_n = 1.\}
    \]
\end{defn}
\begin{defn}
    Let $I\subseteq k[x_0,\dots, x_n]$ be a homogeneous ideal and $t\in \Z^+$. We let $[I]_t$ denote the vector space of $t$-forms in $I$.
\end{defn}
\begin{defn}\label{defn:gamma-a-t}
    Let $\af\subseteq k[x_0,\dots, x_n]$ be a monomial ideal and $t\in \Z^+$. We define $\Gamma(\af, t)$ as the convex hull of $\log([\af]_t)$, and we let $\gamma(\af, t)$ denote the relative interior of $\Gamma(\af, t)$ inside $t\Delta_n$. 
\end{defn}
\begin{remark}
    It is sometimes the case that $\Gamma(\af,t)\subsetneq \Gamma(\af)\cap t\Delta_n$, even if $\af$ is integrally closed. Consider $\af = (x_0,x_1^3)$ as an ideal of $k[x_0,x_1]$; we have $(0.5, 1.5)\in (\Gamma(\af)\cap 2\Delta_1)\setminus \Gamma(\af, 2)$.
\end{remark}
The following proposition shows that Newton polytope of a monomial ideal determines the $F$-pure threshold.
\begin{prop}[\cite{hernandez_f-purity_2016}, Proposition 36]\label{prop:monomial-fpt}
     Let $\af\subseteq  k[x_0,\dots, x_n]$ be a monomial ideal. Then 
     \[
     \fpt(\af) = \frac{1}{\mu}\text{, where }\mu = \inf\{t: t\vec{1}\in \Gamma(\af)\}.
     \]
\end{prop}
Following the proof of \cite{de_fernex_multiplicities_2004}, Theorem 1.4 and the terminology of \cite{mayes_limiting_2014}, we also define the \textit{limiting polytope} of a graded system of monomial ideals.
\begin{defn}
    Let $\af_\bullet$ be a graded system of monomial ideals. That is, suppose $\af_r\af_s\subseteq \af_{r+s}$ for all $r,s\in \Z^+$. We define $\Gamma(\af_\bullet)$ as the closure in $\R^{n+1}$ of the ascending union $\{\frac{1}{2^m}\Gamma(\af_{2^m})\}_{m>0}$. 
\end{defn}
\details{Since we have $\ini_>(f)\ini_>(g) = \ini_>(fg)$, it follows that $\ini_>(I^s)\ini_>(I^t)\subseteq \ini_>(I^{t+s})$, so $\af_\bullet$ as above is indeed a graded system of ideals.}
\subsection{Essential Codimension}
    \begin{defn}[Essential Codimension]\label{defn:ess-dim}
        Let $J\subseteq R = k[x_0,\dots, x_n]$ be a homogeneous ideal. The essential codimension $\ess(J)$ is equal to the minimal $r$ for which there exist linear forms $\ell_1,\dots, \ell_r$ such that $J$ is extended from $I\subseteq k[\ell_1,\dots, \ell_r]$. 
    \end{defn}
    \begin{lemma}\label{lem:ess-dim-sanity-check}
        Let $I, J, \ell_1,\dots, \ell_r$ be as in \Cref{defn:ess-dim}. Then $\ess(I) = r$.
    \end{lemma}
    \begin{proof}The bound $\ess(I)\leq r$ is immediate. Conversely, if $I$ is extended from an ideal $I'\subseteq k[\ell'_1,\dots, \ell'_s]\subseteq k[\ell_1,\dots, \ell_r]$, then $J$ is extended from the same ideal, so $\ess(J)\leq \ess(I)$.
    \end{proof}
\section{Classification of Minimal \texorpdfstring{$F$}{F}-Pure Thresholds}
\subsection{A Bertini Theorem for Essential Codimension}
\begin{convention}
    We identify $\PNV$ with the space of hyperplanes passing through $0\in \A^{n+1}$, as opposed to the usual convention of identifying $\PNV$ with the space of hyperplanes in $\P^n$.
\end{convention}
The following standard lemma relates the condition $\ess(J) < n+1$ to a more familiar condition. 
\begin{lemma}\label{lem:cone-condition}
    Let $k$ be an algebraically closed field, $R = k[x_0,\dots, x_n]$, and $J \subseteq R$ an ideal generated by $d$-forms $f_1,\dots, f_r$. Then $\ess(J)\leq n$ if and only if there exists $p\in \mathbb P^n$ such that $J \subseteq \mf_p^d$.
\end{lemma}
\begin{proof}
    If $\ess(J) \leq n$, then there exist $\ell_1,\dots, \ell_n\in R_1$ such that $J$ is extended from $k[\ell_1,\dots, \ell_n]$. Setting $p = [(\ell_1,\dots, \ell_n)]$, we have $J\subseteq \mf_p^d$. Conversely, suppose $p\in \P^n$ such that $J\subseteq \mf_p^d$ and change coordinates so that $\mf_p = (x_1,\dots, x_n)$. In this case, no monomial summand of the $f_i$ involves $x_0$, so $J$ is extended from $k[x_1,\dots, x_n]$.
\end{proof}
\begin{lemma}\label{lem:maximal-ess-dim-section}
    Let $k$ be an algebraically closed field, $R = k[x_0,\dots, x_n]$, and $J \subseteq R$ a nonzero ideal generated by $d$-forms $f_1,\dots, f_r$. Suppose $\ess(J) = n+1$. Then for general $H\in \PNV$, we have $\ess(J|_H) = n$.
\end{lemma}
\begin{proof}
    Set $Z = \text{Proj}(R/J)\subseteq \P^n$. We define an incidence correspondence as follows:
    \[
    B = \{(z, H)\in Z\times \PNV: z\in H, f_i|_H \in \mf_z^d\text{ for all }1\leq i\leq r\}.
    \]
    Let $p: B\to Z, q: B\to \PNV$ be the projections. Fix $z\in Z$ and change coordinates so that $z = [0:\dots :0:1]$. Write $f_i =: g_i + x_nh_i$ for $g_i\in \mf_z^d, h_i\in \mf^{d - 1}$. Let $(z,H)\in B_z$ where $H = V(\ell)$. 
    
    By definition, there exist $g'_i\in \mf_z^{d}, h'_i\in \mf^{d-1}$ such that $g_i + x_nh_i = g'_i + \ell h'_i$. Write $h'_i =: g''_i + x_nh''_i$, where $g_i''\in \mf_z^{d-1}$. Then $x_n(h_i - \ell h''_i) = g'_i + \ell g''_i - g_i \in \mf_z^d$, so we must have $h_i - \ell h''_i = 0$. In particular, $\ell\mid h_i$. It follows that 
    \[B_z = \{(z, V(\ell): \ell\mid h_i \text{ for all }1\leq i\leq r\}.\] By assumption, $\ess(J) = n+1$. Since $\ess(J) = n+1$, by \Cref{lem:cone-condition} we have $h_i\neq 0$ for some $i$. As each $h_i$ has at most $d-1$ linear factors, we must have $|B_z| \leq d-1 < \infty$.

By the previous paragraph, every closed fiber $B_z$ is zero-dimensional, so $\dim B\leq \dim Z$. Consequently, $\dim q(B)\leq \dim B \leq \dim Z < n$, so $q(B)$ is a proper closed subset of $\PNV$, and so for general $H\in \PNV$, there is no $z\in Z$ such that $(z,H)\in B$. Consequently, there is no $z\in Z\cap H$ such that $f_i \in \mf_z^d|_H$ for all $i$, so another application of \Cref{lem:cone-condition} gives $\ess(J|_H) = n$.
\end{proof}
The following proposition describes the behavior of essential codimension under restriction to a general linear subspace through the origin.
\begin{prop}\label{prop:ess-dim-general-linear-space}
    Let $k$ be an algebraically closed field, $R = k[x_0,\dots, x_n]$, and $J \subseteq R$ a homogeneous ideal. Set $r = \hgt(J)$. Let $L = (\ell_{r+1},\dots, \ell_n)$, where the $\ell_i$ are chosen generally. For $r\leq t\leq n$, set $L_t = (\ell_{t+1}, \dots, \ell_n)$ and $J_t= \frac{J+L_t}{L_t}$. Then for all $r\leq t\leq n$, we have $\ess(J_t) = \min(t+1, \ess(J))$.
\end{prop}
\begin{proof}
By induction, it suffices to consider the case $t = n-1$. The case $\ess(J) = n+1$ is covered by \Cref{lem:maximal-ess-dim-section}; it remains to show that $\ess(J_{n-1}) = \ess(J)$ provided $\ess(J)\leq n$. Set $s = \ess(J)$ and change coordinates so that $J$ is extended from an ideal $I\subseteq k[x_0,\dots, x_{s-1}]$. Suppose $s \leq n$. Let $I' = Ik[x_0,\dots, x_{n-1}]$. By \Cref{lem:ess-dim-sanity-check}, we have $\ess(I') = \ess(I) = \ess(J)$. The isomorphism $k[x_0,\dots, x_n]/(\ell_n) \cong k[x_0,\dots, x_{n-1}]$ identifies $J_{n-1}$ with $I'$, so $\ess(J_{n-1}) = \ess(I') = \ess(J)$. 
\end{proof}
\subsection{An Application of Gr\"unbaum's Inequality}
\begin{defn}
    Let $K\subseteq \R^n$ be a compact set with $\vol(K) > 0$. The \textbf{centroid} $c$ of $K$ is the arithmetic mean of the points of $K$.
    \[
    c = \vol(K)^{-1}\int_{y\in K}ydy.
    \]
\end{defn}
We first recall Gr\"unbaum's inequality, for which we state an equivalent version below. By a \textbf{half-space}, we mean a set of the form
\[
H^+ = H^+(\mathbf{w}, x):=\{\mathbf{v}\in \R^n: \mathbf{v}\cdot \mathbf{w}\geq x\}.
\]
\begin{theorem}[\cite{grunbaum_partitions_1960}, Theorem 2]\label{thm:grunbaum-ineq}
    Let $K\subseteq \R^n$ be a convex body with positive volume and let $c$ denote the centroid of $K$. Let $H^+$ be a half-space whose boundary hyperplane $H$ contains $c$. Then \[\vol(H^+\cap K) \leq \left(1 - \left(\frac{n}{n+1}\right)^n\right)\vol(K).\]
\end{theorem}
\begin{defn}
    We let $\Mmc_n$ denote the quantity $\left(1 - \left(\frac{n}{n+1}\right)^n\right)$ from the theorem.
\end{defn}
For our purposes, we must characterize the equality case of \Cref{thm:grunbaum-ineq}. 
\begin{prop}\label{prop:grunbaum-eq-classification}
    Suppose $H^+,H, K$ are as in \Cref{thm:grunbaum-ineq}, with $\vol(K) > 0$ and $\vol(H^+\cap K) = \Mmc_n\vol(K)$. Let $H$ denote the boundary hyperplane of $H^+$. Then there exists a convex body $K'\subseteq H^+\cap K$ and a point $q\in K\setminus H^+$ such that $K'$ is contained in a hyperplane parallel to $H$ and $K = \conv(K'\cup \{q\})$.  
\end{prop}
\begin{proof}
    Follows from \cite{myroshnychenko_grunbaums_2018}, Corollary 8.
\end{proof}
\begin{defn}\label{defn:pi_n-and-T_n}
    Let $\pi_n: \R^{n+1}\to \R^n$ denote the projection onto the first $n$ coordinates and let $T_n:= \pi_n(\Delta_n)$.
\end{defn}
\begin{lemma}\label{lem:grunbaum-eq-for-simplices}
   Let $z_n = (\frac{1}{n+1},\dots, \frac{1}{n+1})$ denote the centroid of $T_n$. Let $H^+$ be a half-space whose boundary hyperplane $H$ contains $z_n$. Then \[\vol(H^+\cap T_n) \leq \frac{\Mmc_n}{n!}\] with equality if and only if $H$ is parallel to a facet $F$ of $T_n$ with $F\subseteq H^+$.
\end{lemma}
\begin{proof}
    If $K'$ is an $n-1$-dimensional convex set and $q$ a point not contained in the hyperplane supporting $K'$ such that $\conv(K'\cup \{q\})$ is a polytope, then $K'$ is a facet of $\conv(K'\cup \{q\})$. The result therefore follows from \Cref{prop:grunbaum-eq-classification}.
\end{proof}
We recall the following standard fact from convex analysis:
\begin{lemma}[\cite{Rockafellar_1970}, Corollary 11.6.1]\label{lem:supporting-hyperplane}
    Let $K\subseteq \R^n$ be a convex set and $x\in \del K$. Then there exists a half-space $H^+$ such that $K\subseteq H^+$ and such that $x\in \del H^+$.
\end{lemma}
\begin{lemma}\label{lem:vol-upper-bound}
    Let $P\subseteq T_n$ be a closed convex set with $z_n$ not in the interior of $P$. Then $\vol(P) \leq \Mmc_n/n!$ with equality if and only if $P$ is the intersection of $T_n$ with a half-space satisfying the conditions of \Cref{lem:grunbaum-eq-for-simplices}.
\end{lemma}
\begin{proof}
    If $z_n\notin \del P$, then for $0 < \varepsilon < \text{dist}(P, z_n)$, the set $\{x\in T_n: \text{dist}(x,P) \leq \varepsilon\}$ is a strictly larger convex set which does not contain $z_n$ in its interior. We may therefore assume $z_n\in \del P$. By \Cref{lem:supporting-hyperplane}, we may replace $P$ by $H^+\cap T_n$, where $H^+$ is the half-space containing $P$ with $z_n\in \del H^+$. In this case, the result is immediate from \Cref{lem:grunbaum-eq-for-simplices}.
\end{proof}
\details{To see that $P_\varepsilon:=\{x\in T_n: \text{dist}(x, P)\leq \varepsilon\}$ is convex, note that $P_\varepsilon$ is the Minkowski sum of $P$ with the closed ball of radius $\varepsilon$, and the Minkowski sum of two convex sets is convex.}
\subsection{Proof of Theorem \texorpdfstring{\ref{thm:equigen-fpt}}{A}}
To start, we recall a theorem of Rees.
\begin{theorem}[\cite{huneke_integral}, Proposition 11.2.1, Theorem 11.3.1]\label{thm:Rees}
    Let $(R,\mf)$ be a formally equidimensional local ring and $I\subseteq J$ two $\mf$-primary ideals. Then $e(I) = e(J)$ if and only if $\bar{I} = \bar{J}$. 
\end{theorem}
The conclusion of \Cref{thm:Rees} also holds when $R$ is a polynomial ring, $\mf\subseteq R$ a maximal ideal, and $I,J$ two $\mf$-primary ideals. As a consequence, we may restate the conclusion of \Cref{thm:equigen-fpt} in terms of essential dimension.
\begin{lemma}\label{lem:ess-dim-equals-height}
    Let $R = k[x_0,\dots, x_n]$ and $I\subseteq R$ an ideal generated by $d$-forms. If $\hgt(I) = h$, then the following are equivalent:
    \begin{enumerate}[(i)]
        \item  $\ess(I) = h$.
        \item $\overline{I} = (x_0,\dots, x_{h-1})^d$ up to change of coordinates. 
        \item $I\subseteq (x_0,\dots, x_{h-1})^d$ up to change of coordinates. 
    \end{enumerate}
\end{lemma}
\begin{proof}\mbox{}
\begin{itemize}
    \item[(i)$\implies$(ii):] Let $R' = k[x_0,\dots, x_{h-1}]$. Without loss of generality, there exists an ideal $I'$ of $R'$ such that $I = I'R$. By flatness of $R'\hookrightarrow R$, we have $\hgt(I') = h$. Let $\mf'$ denote the homogeneous maximal ideal of $R'$. Let $S' = \overline{k}[x_0,\dots, x_{h-1}]$. If $J'\subseteq I'S'$ is generated by $h$ general $d$-forms in $I'S'$, then $J'$ is a $(d,\dots, d)$-complete intersection of height $h$. We have $e(J') = d^h = e((\mf')^d)$, hence $\overline{J'} = \overline{(\mf'S')^d}$ by \Cref{thm:Rees}. As $I\subseteq (\mf'S')^d\subseteq \overline{I}$ and $(\mf'S')^d$ is integrally closed, we have $\overline{I'} = (\mf'S')^d$. By faithful flatness of $R'\to S'$, we have $\overline{I'} = \overline{I'S'}\cap R' = (\mf')^d$. Passing back to the ring $R$, the ideal $(\mf')^dR$ is integrally closed and $I\subseteq (\mf')^dR = \overline{I'}R \subseteq \overline{I}$, hence $\overline{I} = (\mf')^dR$.
    \item[(ii)$\implies$(iii):] This follows from the containment $I\subseteq \overline{I}$.
    \item[(iii)$\implies$(i):] This follows from the argument of \Cref{lem:cone-condition}. \details{To be precise, we note that every $d$-form in $(x_0,\dots, x_{h-1})^d$ cannot involve $x_h,\dots, x_n$, and so the generators of $I$ are contained in $k[x_0,\dots, x_{h-1}]$.}
\end{itemize}
\end{proof}
We state the main technical lemma of this section.
\begin{lemma}\label{lem:min-fpt-char-CI-height-n}
    Let $k$ be an algebraically closed field of characteristic $p>0$ and let $R = k[x_0,\dots, x_n]$. Let $I = (f_1,\dots, f_n)\subseteq R$ denote a complete intersection ideal generated by $d$-forms. Then $\fpt(I) = n/d$ if and only if $\ess(I) = n$.
\end{lemma}
We begin our proof of \Cref{lem:min-fpt-char-CI-height-n} by computing the Hilbert series of $R/I^s$.
\begin{lemma}\label{lem:hilb-fn-power-of-CI}
    Let $I,R$ be as in \Cref{lem:min-fpt-char-CI-height-n}. For a positive integer $s$ and for $t \geq (d-1)n + d(s-1)$, we have $H_R(R/I^s, t) = \binom{n+s-1}{n}d^n$. 
\end{lemma}
\begin{proof}
    We define 
    \[\Lc_{n,s}:= \{(a_1,\dots, a_n): a_i\geq 0, a_1 + \dots + a_n \leq s-1\}.\]
    By \cite{guardo_powers_2005}, Corollary 2.3, we have 
    \begin{equation}\label{eqn:hilb-fn-formula}
    H_R(R/I^s, t) = \sum_{(a_1,\dots, a_r)\in \Lc_{n,s}}H_R(R/I, t - d(a_1+\dots +a_n)).\end{equation}
    The Hilbert series of $R/I$ is given by 
    \begin{equation}\label{eqn:hilb-series-of-CI}
    \sum_{i\geq 0}H_R(R/I,i)t^i =\frac{(1-t^d)^n}{(1-t)^{n+1}} = (1+t+\dots + t^{d-1})^n(1+t+t^2+t^3 + \dots),
    \end{equation}
    hence $H_R(R/I, t) = d^n$ for $t\geq (d-1)n$. By \Cref{eqn:hilb-fn-formula,eqn:hilb-series-of-CI}, for $t \geq (d-1)n + d(s-1)$ we have 
    \[H_R(R/I^s, t) = |\Lc_{n,s}|d^n = \binom{n + s - 1}{n}d^n.\]
\end{proof}
In particular, \Cref{lem:hilb-fn-power-of-CI} holds for $t\geq d(s+n)$.
\begin{lemma}\label{lem:fpt-relint-bound}
    Let $\af\subseteq R$ be a monomial ideal containing a monomial $m$ of degree $t$. For any $t'>t$, if $\frac{t'}{n+1}\vec 1\in \gamma(\af, t')$, then $\fpt(\af) > \frac{n+1}{t'}$.
\end{lemma}
\begin{proof}
    Set $y = \log(m)$. By convexity of $\Gamma(\af)$, we have $\l y + (1-\l)\gamma(\af, t')\subseteq \Gamma(\af)$ for all $\l \in [0,1]$. Taking $0 < \l \ll 1$, we obtain $\frac{\l t + (1-\l)t'}{n+1}\vec{1}\in \Gamma(\af)$, which implies $\fpt(\af)\geq \frac{n+1}{\l t + (1-\l)t'} > \frac{n+1}{t'}$ by \Cref{prop:monomial-fpt}.
\end{proof}
\details{Choose $\varepsilon>0$ such that for all $x\in t'\Delta_n$ with $|\frac{t'}{n+1}\vec{1} - x| < \varepsilon$, we have $x\in \gamma(\af, t')$. Write $y = (y_0,\dots, y_n)$. Choose $\lambda < \frac{\varepsilon}{2t(n+1)}$. For $0\leq i\leq n$, set 
\[
x_i = \frac{t'}{n+1}+\left(\frac{\l}{\l-1}\right)\left(y_i-\frac{t}{n+1}\right).
\]
Set $x = (x_0,\dots, x_n)$. As $0\leq y_i\leq t$ for all $i$ and $\lambda \leq \frac{1}{2}$, by the triangle inequality we have \[\text{dist}\left(\frac{t'}{n+1}\vec{1}, x\right) \leq \sum_{i=0}^n \left(\frac{\l}{\l-1}\right)\left(y_i-\frac{t}{n+1}\right) \leq 2\l(n+1)t < \varepsilon,\]
so $x\in \gamma(\af, t')\subseteq \Gamma(\af)$. As $\l y + (1-\l)x = \frac{\l t+(1-\l)t'}{n+1}\vec{1}\in \Gamma(\af)$.
}
Lastly, we need a result relating volume and integer point counts for convex bodies.
\begin{lemma}\label{lem:vol-lower-bound}
    Let $\Delta_n, T_n, \pi_n$ be as in \Cref{defn:simplex,defn:pi_n-and-T_n}. For $t\in \Z^+$ and $P\subseteq t\Delta_n$ a convex set, we have 
    \begin{equation}\label{eqn:lattice-pt-vol}
    |\vol_n(\pi_n(P)) -  \#(P\cap \Z^{n+1})| \leq \sum_{i=1}^{n-1}{n\choose i}\frac{t^i}{i!}.\end{equation}
\end{lemma}
\begin{proof} 
Since $t\Delta_n$ is contained in the affine space $x_0 + \dots + x_n = t$ and $t\in \Z^+$, $\pi_n$ induces a bijection between $t\Delta_n\cap \Z^{n+1}$ and $T_n\cap \Z^n$, so \Cref{eqn:lattice-pt-vol} can be interpreted as a statement relating the volume and integer point count of $\pi_n(P)$. For each $P'$ occurring as an $i$-dimensional projection of $\pi_n(P)$ onto an $i$-dimensional coordinate axis, $\frac{1}{t}P'$ is contained in an $i$-dimensional simplex, so we have $\vol_i(P')\leq \frac{t^i}{i!}$. The result then follows from \cite{davenport_principle_1951}.\end{proof}
\details{In our application, $P$ is convex, so $h = 1$.
}
We now prove \Cref{lem:min-fpt-char-CI-height-n}.
\begin{proof}
    Let $\pf$ be a minimal prime over $I$. Since $k = \bar{k}$ and $I$ is homogeneous, we may change coordinates so that $\pf = (x_1,\dots, x_n)$. Let $>$ denote the lexicographic order and define the graded system of ideals $\af_\bullet = \{\ini_>(I^{nm})\}_m$. Since $\pf^r$ is a monomial ideal for all $r\geq 0$ and $I^r\subseteq \pf^r$, we have $\af_m\subseteq \pf^{nm}$ for all $m\geq 0$. Since $\af_\bullet$ is graded, we have for any $t\in \Z^+$
    \[
    [\af_{2^m}]_{2^mt}[\af_{2^m}]_{2^mt} \subseteq [\af_{2^m} \af_{2^m}]_{2^{m+1}t} \subseteq  [\af_{2^{m+1}}]_{2^{m+1}t}.
    \]
    It follows that $\{\frac{1}{2^m}\gamma(\af_{2^m}, 2^mt)\}_m$ is an ascending chain of convex subsets of $H_t$. We then set $t = d(n+1)$ and let $\Pc$ denote the ascending union $\bigcup_{m\geq 1}\gamma(\af_{2^m}, 2^md(n+1))$. If $d\vec 1\in \Pc$, there exists some $m$ such that $d\vec 1\in \gamma(\af_{2^m}, 2^md(n+1))$. By \Cref{lem:fpt-relint-bound}, we have $\fpt(\af_{2^m}) > \frac{n+1}{2^md(n+1)} = \frac{1}{2^md}$, so $\fpt(I) > n/d$ by \Cref{prop:fpt-properties}.

    Conversely, suppose $d\vec 1\notin \Pc$. Then for all $m$, we have $d\vec 1\notin \frac{1}{2^m}\gamma(\af_{2^m}, 2^md(n+1))$. By \Cref{lem:vol-upper-bound}, we have 
    \begin{equation}
        \vol(\Pc) = \lim_{m\to\infty}\vol\left(\frac{1}{2^m}\gamma(\af_{2^m}, 2^md(n+1))\right) \leq (d(n+1))^n \frac{\Mmc_n}{n!}.
    \end{equation}
    We now derive a lower bound for $\vol(\Pc)$. First, by \Cref{lem:hilb-fn-power-of-CI}, we have
    \begin{align*}
        \#\Z^{n+1}\cap (\gamma(\af_{2^m}, 2^md(n+1))) & \geq H_R(\af_{2^m}, 2^md(n+1)) 
        \\&= H_R(I^{2^mn}, 2^md(n+1)) 
        \\ &= \binom{n + 2^md(n+1)}{n} - \binom{n + 2^mn - 1}{n}d^n
    \end{align*}
    provided $2^md(n+1) \geq d(2^mn + n)$, which is satisfied for all $m\geq \log_2 n$. Using the approximation $\binom{a + b}{b} = \frac{a^b}{b!} + O_b(a^{b-1})$, we have
    \begin{equation}\label{eqn:lattice-pt-est}
    \#\Z^{n+1}\cap (\gamma(\af_{2^m}, 2^md(n+1))) \geq \frac{(2^md)^n}{n!}\left((n+1)^n - n^n\right) + O_n(2^{m(n-1)})
    \end{equation}
    Combining the bounds \Cref{lem:vol-lower-bound,eqn:lattice-pt-est}, we have
    \begin{align*}
        \vol(\Pc) &= \lim_{m\to\infty}\vol\left(\frac{1}{2^m}\gamma(\af_{2^m}, 2^md(n+1))\right) = \lim_{m\to\infty}\frac{1}{2^{mn}}\vol(\gamma(\af_{2^m}, 2^md(n+1)))\\
        &\geq \lim_{m\to\infty} \frac{1}{2^{mn}}\left(\#(\Z^{n+1}\cap \gamma(\af_{2^m},2^md(n+1)) - \sum_{i=1}^{n-1}{n\choose i}\frac{(2^md(n+1))^i}{i!} \right)
        \\& \geq \lim_{m\to\infty} \frac{1}{2^{mn}}\left(\frac{(2^md)^n}{n!}\left((n+1)^n - n^n\right) + O_n(2^{m(n-1)})\right)\\&= (d(n+1))^n\frac{\Mmc_n}{n!}.
    \end{align*}
    It follows that $\vol(\Pc) = \vol(\overline{\Pc})= (d(n+1))^n\frac{\Mmc_n}{n!}$, so by \Cref{lem:grunbaum-eq-for-simplices}, we have $\overline{\Pc}=H^+\cap (d(n+1))\Delta_n$. Moreover, the boundary hyperplane $H$ of $H^+$ is parallel to a facet $F$ of $(d(n+1))\Delta_n$ with $F\subseteq H^+$ and $d(n+1)\eta_n\in H$. 

    For $\a\in \R$, define $D_{t,\b} = \{(a_0,\dots, a_n)\in t\Delta_n: a_0 \leq \b\}$. Since $\af_m\subseteq \pf^{mn}$, for any monomial $x_0^{a_0}\dots x_n^{a_n}\in (\af_m)_t$, we have $a_1+\dots+a_n \geq mn$ and hence $a_0\leq t - mn$. In particular, for all $m\geq 0$ we have
    \[
    \gamma(\af_{2^m}, 2^md(n+1))\subseteq D_{2^md(n+1), 2^md(n+1) - 2^mn}.
    \]
    As a consequence, we conclude $\Pc\subseteq D_{d(n+1), d(n+1) - n}$. As $F\subseteq \Pc$, the only possible choice for $F$ is the facet $\{a_0 = 0\}\subseteq d(n+1)\Delta_n$. We conclude that $\overline{\Pc} = D_{d(n+1), d}$. We then have \[\Gamma(\af_1, d(n+1))\subseteq \overline{\Pc} = D_{d(n+1),d} = \Gamma(\pf^{nd}, d(n+1)),\] so $[\af_1]_{d(n+1)}\subseteq [\pf^{nd}]_{n(d+1)}.$ For each generator $f_i$ of $I$, we have $x_0^d\ini_>(f_i^n)\in [\af_1]_{d(n+1)}\subseteq [\pf^{nd}]_{n(d+1)}$, so $x_0\nmid \ini_>(f_i^n)$ for all $i$. As $\ini_>(f_i^n) = \ini_>(f_i)^n$, we deduce that $I\subseteq \pf^d$. By \Cref{lem:ess-dim-equals-height}, we have $\ess(I) = n$.
\end{proof}
We are now able to prove \Cref{thm:equigen-fpt}. By \Cref{lem:ess-dim-equals-height}, it suffices to prove the following.
\begin{theorem*}[\ref{thm:equigen-fpt}]
    Let $k$ be an algebraically closed field of characteristic $p>0$. Let $I$ be a homogeneous ideal in $k[x_0,\dots, x_n]$ generated by $d$-forms and set $h = \hgt(I)$. Then $\fpt(I) = h/d$ if and only if $\ess(I) = h$.
\end{theorem*}
\begin{proof}
    Let $k$ be an algebraically closed field and $R = k[x_0,\dots, x_n]$. Let $I\subseteq R$ be an ideal generated by $d$-forms, and suppose that $\hgt(I) = n, \fpt(I) = n/d$. If $f_1,\dots, f_n$ are $n$ general $d$-forms in $I$, then $J = (f_1,\dots, f_n)$ is a complete intersection. By \Cref{prop:tw-fpt-equigen} and \Cref{prop:fpt-properties} (i), we have
    \[
    n/d \leq \fpt(J)\leq \fpt(I) = n/d.
    \]
    By \Cref{lem:min-fpt-char-CI-height-n,lem:ess-dim-equals-height}, we may change coordinates on $R$ such that $\bar{J} = (x_1,\dots, x_n)^d$. Then we have $(x_1,\dots, x_n)^d\subseteq \bar{I}$. Let $>$ denote the lexicographic order, and let $g$ be a $d$-form in $\bar{I}$. Write $\ini_>(g) = x_0^{a_0}\cdots x_n^{a_n}$, and note that $(x_1,\dots, x_n)^d\subseteq \ini_>(\overline{I})$. Set $a = \max_i a_i$. Then
    \[
    \ini_>(g)^{\floor{(p^e-1)/a}}\prod_{i=1}^{n}(x_i^d)^{\floor{\left((p^e-1) - a_i\floor{(p^e-1)/a}\right)/d}}\notin \mf^{[p^e]},
    \]
    so we have
    \begin{align*}
\lim_{e\to\infty}\frac{\nu_{\ini_>(\bar{I})}}{p^e}&\geq \lim_{e\to\infty}\frac{1}{p^e}\left(\floor{\frac{p^e-1}{a}} + \sum_{i=1}^n\floor{\frac{p^e-1}{d} - \frac{a_i\floor{(p^e-1)/a}}{d}}\right)\\=&\frac{1}{a} + \sum_{i=1}^{n} \left(\frac{1}{d} - \frac{a_i}{ad}\right) = \frac{n}{d} + \frac{a_0}{ad}.
    \end{align*}
    Consequently, by \Cref{prop:fpt-properties,prop:semicont} we have
    \[
    \frac{n}{d} = \fpt(I) = \fpt(\bar{I}) \geq \fpt(\ini_>(\bar{I})) \geq \fpt((x_1,\dots,x_n)^d + (x_0^{a_1}\cdots x_n^{a_n})) = \frac{n}{d} + \frac{a_0}{ad},
    \]
    so $a_0 = 0$ and hence $\ini_>(g)\in (x_1,\dots, x_n)^d$. As $>$ is the lexicographic order, it follows that $g\in (x_1,\dots, x_n)^d$. As $g$ was arbitrary, we conclude that $I \subseteq (x_1,\dots, x_n)^d$, hence $\overline{I} = (x_1,\dots, x_n)^d$ by \Cref{lem:ess-dim-equals-height}.

    Next, we consider the case that $\hgt I\neq n$. If $\hgt I = n+1$, then $\overline{I} = (x_0,\dots, x_n)^d$ by \Cref{thm:Rees}. Otherwise, suppose $\hgt I = h \leq n-1$. Let $L$ be an ideal generated by $n-h$ linear forms. Then $\frac{h}{d}\leq \fpt(\frac{I + L}{L})\leq \frac{h}{d}$, so by the height-$n$ case, we have $\ess(\frac{I+L}{L}) = h$. By \Cref{prop:ess-dim-general-linear-space}, the same holds for $I$. 
\end{proof}

\section{The Test Ideal at the Threshold}
In the introduction, we claimed that the best-known result in characteristic zero (\Cref{thm:equigen-lct}) is stronger than the previous best-known result in positive characteristic (\Cref{prop:tw-fpt-equigen}). Indeed, \Cref{thm:equigen-lct} shows that analogs of \Cref{prop:tw-fpt-equigen} and \Cref{thm:equigen-fpt} hold in characteristic zero. 
\begin{prop}
Let $k$ be an algebraically closed field of characteristic zero. Let $I$ be a homogeneous ideal in $k[x_0,\dots, x_n]$ generated by $d$-forms. If $h$ is the height of $I$, then $\fpt(I) \geq h/d$ with equality if and only if $\overline{I} = (x_0,\dots, x_{h-1})^d$ up to change of coordinates.
\end{prop}
\begin{proof}
    Let $R = k[x_0,\dots, x_n]$. Since $(R_\pf, (1)^t)$ is klt for any prime ideal $\pf\in \Spec R$ and all $t > 0$, the non-klt locus $Z$ of $(R,  I^{\lct(I)})$ is contained in $V(I)$. Consequently, by \Cref{thm:equigen-lct} we have
    \[
    \lct(I) \geq \frac{\codim Z}{d} \geq \frac{\hgt I}{d} = \frac{h}{d}.
    \]
    Write $I = (f_1,\dots, f_r)$.
    If additionally $\lct(I) = \frac{h}{d}$, then $\codim(Z) = h$ and $\lct(I) = \frac{\codim(Z)}{d}$, so there exist linear forms $\ell_1,\dots, \ell_h \in R$ such that $f_i\in k[\ell_1,\dots, \ell_h]$. Changing coordinates, we may assume $\ell_i = x_{i-1}$ for $1\leq i\leq h$. The result then follows from \Cref{lem:ess-dim-equals-height}.
\end{proof}
By \cite{hara_generalization_2003}, the correct positive-characteristic analog of \Cref{thm:equigen-lct} considers strong $F$-regularity and the $F$-pure threshold. We direct the reader to \cite{hara_on_a_2004} for background on the test ideal $\tau(R,\af^t)$, which cuts out the non-strongly $F$-regular locus of the pair $(R, \af^t)$. With this in mind, we are now able to give an example of the failure of \Cref{thm:equigen-lct} to generalize to positive characteristic.
\begin{example}\label{ex:naive-analog-failure}
    Suppose $p\equiv 2\mod 3$. Let $R = \F_p[x,y,z]$ and $f = (x^3 + y^3 + z^3)$. By \cite[Theorems 3.1 and 3.3]{hernandez-diagonal-2015}, we have $\fpt(f) = 1 - \frac{1}{p}$ and $\tau(R, f^{1-1/p}) = (x,y,z)$. A naive translation of \Cref{thm:equigen-lct} predicts that $\fpt(f)\geq \frac{\hgt((x,y,z))}{3}$, but this is not the case.
\end{example}
Motivated by the failure of the positive-characteristic analog of \Cref{thm:equigen-lct} in the case that $p$ divides the denominator of $\fpt(f)$, we impose the additional condition that the pair $(R, I^{\fpt(I)})$ is sharply $F$-split.
\begin{theorem*}[\ref{thm:test-ideal-at-threshold}]
    Suppose $\char k = p>0$ and $R = k[x_0,\dots, x_n]$. Let $I\subseteq R$ be an ideal generated by $d$-forms and let $c = \fpt(I)$. Let $h$ denote the height of $\tau(R, I^c)$. If $(R, I^c)$ is sharply $F$-pure, then $c\geq \frac{h}{d}$.
\end{theorem*}
\begin{proof}
    Let $\pf$ be a minimal prime over $\tau(I^c)$. Let $c_\pf:= \fpt(R/\pf, \mf/\pf)$. For any $\varepsilon>0$, the pair $(R/\pf, (\mf/\pf)^{c_\pf+\varepsilon})$ is not $F$-pure by \cite[Proposition 2.2 (4)]{takagi_f-pure_2004}, hence \cite[Lemma 3.9]{takagi_inversion_2004}) yields that for all $e \gg 0$, we have
    \[
    \mf^{\ceil{p^e(c_\pf+\varepsilon)}}(\pf^{[p^e]}:\pf)\subseteq \mf^{[p^e]},
    \]
    hence 
    \begin{equation}\label{eqn:fedder}
    (\pf^{[p^e]}:\pf)\subseteq (\mf^{[p^e]}:\mf^{\ceil{p^e(c_\pf+\varepsilon)}}) = \mf^{[p^e]}+\mf^{(n+1)(p^e-1)+1-{\ceil{p^e(c_\pf+\varepsilon)}}},\end{equation}
    where the second equality is by \cite[Lemma 3.2]{bhatt_f_2015}.

    By \cite[Propositions 4.5 and 4.7]{schwede_centers_2010}, the ideal $\pf$ is uniformly $(I^c, F)$-compatible, hence for all $e > 0$ we have $I^{\ceil{(p^e-1)c}}\subseteq (\pf^{[p^e]}:\pf)$. On the other hand, because $(R, I^c)$ is sharply $F$-pure, there exists $e>0$ such that for all $f > 0$, we have $I^{\ceil{c(p^{ef}-1)}}\not\subseteq \mf^{[p^{ef}]}$. Choose a minimal generator $z$ of $I^{\ceil{c(p^{ef}-1)}}$ which is not in $ \mf^{[p^{ef}]}$. Estimating $\deg(z)$ with \Cref{eqn:fedder} yields
    \begin{equation}\label{eqn:degree-comparison}
    d\ceil{c(p^{ef}-1)} \geq (n+1)(p^{ef}-1)+1-{\ceil{p^{ef}(c_\pf+\varepsilon)}}.
    \end{equation}
    Sending $f\to\infty$, we conclude that
    \begin{align*}
    cd =& \lim_{f\to\infty} \frac{d\ceil{c(p^{ef}-1)}}{p^{ef}} \\\geq &\lim_{f\to\infty} \frac{(n+1)(p^{ef}-1)+1-{\ceil{p^{ef}(c_\pf+\varepsilon)}}}{p^{ef}} =(n+1)-c_\pf-\varepsilon.
    \end{align*}
    By \cite[Proposition 2.6]{takagi_f-pure_2004}, we have
    \begin{equation}\label{eqn:height-fpt}
       \fpt(R/\pf, \mf/\pf) \leq \dim(R/\pf)\leq n+1-h
    \end{equation}
    Sending $\varepsilon\to 0$, we conclude that $cd \geq h$, proving the claim.
\end{proof}
We can now say a bit more about the case of equality.
\begin{prop}\label{prop:test-ideal-in-equality-case}
    Assume the setup of \Cref{thm:test-ideal-at-threshold}. If $c = h/d$, then there exist linear forms $\ell_1,\dots, 
    \ell_h$ such that $\tau(I^c) = (\ell_1,\dots, \ell_h)$.
\end{prop}
\begin{proof}
    Let $\pf_1,\dots, \pf_r$ denote the minimal primes over $\tau(I^c)$. Suppose $c = h/d$ and run the argument of \Cref{thm:test-ideal-at-threshold} with $\pf = \pf_i$. Each inequality must be an equality; in particular, we must have $c_\pf = n + 1 -h = \dim(R/\pf)$. By \cite[Theorem 2.7]{takagi_f-pure_2004}, $R/\pf$ is regular, hence $\pf_i$ is generated by linear forms. 

    Next, we show that $r = 1$. Suppose to the contrary that $\pf_1,\pf_2$ are distinct minimal primes over $\tau(I^c)$. Let $1\leq j\leq h$ and let \[u_1,\dots, u_j,v_1,\dots, v_j,w_{j+1},\dots, w_h\] be linear forms such that $(u_1,\dots, u_j,v_1,\dots, v_j)$ are linearly independent and 
    \[\pf_1 = (u_1,\dots, u_j,w_{j+1},\dots, w_h)\qquad \pf_2 = (v_1,\dots, v_j,w_{j+1},\dots, w_h).\]
    For $q$ a power of $p$, we have
    \[
    (\pf_1^{[q]}:\pf_1)\cap (\pf_2^{[q]}:\pf_2) = \pf_1^{[q]}+\pf_2^{[q]} + (u_1\dots u_j v_1\dots v_j w_{j+1}\dots w_h)^{q-1}R.
    \]
    Putting $q = p^{ef}$ as in the proof of \Cref{thm:test-ideal-at-threshold} and noting that 
    \[I^{\ceil{(p^{ef}-1)(h/d)}}\subseteq \left((\pf_1^{[q]}:\pf_1)\cap (\pf_2^{[q]}:\pf_2)\right)\setminus \mf^{[p^{ef}]},\]
    comparing degrees gives us the contradiction
    \[
    h(p^{ef}-1) \geq (h+j)(p^{ef}-1) > h(p^{ef}-1).
    \]
    It follows that $r = 1$. Moreover, by \cite[Corollary 3.3]{schwede_centers_2010}, $\tau(I^c)$ is a radical ideal, so $\tau(I^c) = \pf_1$ is generated by linear forms.
\end{proof}
Although $\tau(I^c)$ is generated by linear forms $\ell_1,\dots, \ell_h$ in the $c = h/d$ case, it needn't be the case that $I$ is extended from $k[\ell_1,\dots, \ell_h]$. By \Cref{thm:equigen-lct}, an example of the failure of this property must necessarily satisfy $\fpt(I) \neq \lct(I)$ and $(p^e-1)\fpt(I)\in \Z^+$ for some $e>0$. A good source of such examples is \cite{canton_fpt_2016}.
\begin{example}\label{ex:CHSW}
    We homogenize the example of Proposition 2.7 in op. cit. with a minor variation. Let $p$ be a prime number and choose $n\geq 4$, $q = p^e\geq n$ such that
    $p\nmid (n(q-1) -q -1) =: s$. We then set $c = 1/(q-1)$ and 
    \[
    R = \F_p[x_1,\dots, x_n, y_1,\dots, y_s],\quad f = y_1\dots y_s(x_1^{q+1}+\dots+x_n^{q+1}) + (x_1\dots x_n)^{q-1}.
    \]
    Put $\bfr = (x_1,\dots, x_n)$. As $f^1\in (\bfr^{[q]}:\bfr)\setminus \mf^{[q]}$, proposition 2.3 in op. cit.--- which is stated for the local case, but holds in the graded case by the same argument --- implies that $\fpt(f) = c$.
    The argument in loc. cit. also implies that $(R, f^c)$ is sharply $F$-pure and $\tau(R, f^c)\subseteq \bfr$. As $\tau(R, f^c)$ is radical by \cite[Corollary 3.3]{schwede_centers_2010}, we may write $\tau(R, f^c) = \bfr\cap \qf_1\cap \dots \cap \qf_r$, where each $\qf_i\neq \bfr$ is a minimal prime over $\tau(R, f^c)$.
    
    We now show $\tau(R, f^c) = \bfr$. To see this, set \[
    Y = y_1\dots y_s,\quad R' =R[Y^{1/s}, Y^{-1}],\quad f'=x_1^{q+1}+\dots +x_n^{q+1} + (x_1\dots x_n)^{q-1}.
    \]
    The automorphism of $R'$ given by $x_i\mapsto Y^{1/s}x_i$ sends $f$ to $Y^{1 + (q+1)/s}f'$, so $\tau(R', f^c)= \tau(R', (f')^c)$. Because $p\nmid s,$ the extension $R\to R'$ is etale, so we have
    \begin{equation}\label{eqn:dehomogenized-test-ideal}
    \bfr R'\cap \qf_1R' \cap \dots \cap \qf_r R' = \tau(R, f^c)R' = \tau(R', f^c) = \tau(R', (f')^c) = \bfr R',
    \end{equation}
    where the final equality is by \cite[Remark 2.6]{canton_fpt_2016}, which implies that the Jacobian ideal of $f'$ is primary to $\bfr$. By \Cref{eqn:dehomogenized-test-ideal} we must have $\qf_i R' = R'$ for all $1\leq i\leq r$, hence $\qf_i\in V(y_1\dots y_s)$. We aim to show that $r=0$. To see this, note that by \cite[Proposition 4.5]{schwede_centers_2010}, the primes $(x_1,\dots, x_n), \qf_1,\dots, \qf_r$ are precisely the prime ideals which are uniformly $(f^c, F)$-compatible, hence we have $f^{\ceil{c(q-1)}}\in (\qf_i^{[q]}:\qf_i)$. As $\qf_i\in V(y_1\dots y_s)$, we have $y_j\in \qf_i$ for some $j$. We may then write $\qf_i = (y_j, Q)$ where $Q$ is extended from the subring $k[y_1,\dots, \hat{y_j},\dots, y_s,x_1,\dots, x_n]$. We compute (c.f. \cite[Lemma 5.2]{deStefani_defect_2026})
    \begin{align*}
    (\qf_i^{[q]}: \qf_i) &= (\qf_i^{[q]}:y_j) \cap (\qf_i^{[q]}:Q)\\&= (Q^{[q]},y_j^{q-1}) \cap (y_j^q, Q^{[q]}:Q) = \qf_i^{[q]} + y_j^{q-1}(Q^{[q]}:Q).
    \end{align*}
    The $y_j$-degree of $f^{\ceil{c(q-1)}}$ is strictly less than $q-1$, so in fact \(f^1 = f^{\ceil{c(q-1)}} \in \qf_i^{[q]}.\) This contradicts the earlier observation that $f^1\notin \mf^{[q]}$, so $r = 0$ and $\tau(R, f^c) = \bfr$.

    In this example, we have observed that
    \[\fpt(f) = \frac{1}{q-1} = \frac{n}{n(q-1)}=\frac{\hgt(\tau(f^{\fpt(f)}))}{\deg(f)},\]
    but $f\notin k[x_1,\dots, x_n]$ in contrast to \Cref{thm:equigen-lct}. 
    
    Indeed, suppose $f$ is contained in a subring $R_0:= k[\ell_1,\dots, \ell_t]\subseteq R$, where the $\ell_i$ are linear forms. By observing that $\frac{\del \ell_i}{\del x_j} \in k, \frac{\del \ell_i}{\del y_j} \in k$ and repeatedly applying the Leibniz rule, we see that $R_0$ is closed under the operators $\del/\del y_i, \del/\del x_i$. Additionally, pick linear forms $\ell_{t+1},\dots, \ell_{s+n}$ such that $R = R_0[\ell_{t+1},\dots, \ell_{s+n}]$. If $g_1,\dots,g_m\in R$ such that $g_1\dots g_m\in R_0\setminus \{0\}$, then each of the $g_i$ must be constant in the variables $\ell_{t+1},\dots, \ell_{s+n}$, so $g_i\in R_0$. For any $1\leq i\leq n, 1\leq j\leq s$, we may pick $j'\neq j$ and compute
    \[
    R_0\ni\frac{\del^2f}{\del x_i\del y_{j'}} = 
    y_1\dots \widehat{y_{j'}} \dots y_s x_i^q,
    \]
    hence $x_i,y_j\in R_0$. It follows that $R_0 = R$.

\end{example}
\printbibliography

@book{Rockafellar_1970,
    AUTHOR = {Rockafellar, R. Tyrrell},
     TITLE = {Convex analysis},
    SERIES = {Princeton Mathematical Series},
    VOLUME = {No. 28},
 PUBLISHER = {Princeton University Press, Princeton, NJ},
      YEAR = {1970},
     PAGES = {xviii+451},
   MRCLASS = {26.52 (46.00)},
  MRNUMBER = {274683},
MRREVIEWER = {Ky\ Fan},
}

@article {deStefani_defect_2026,
    AUTHOR = {De Stefani, Alessandro and N\'u\~nez-Betancourt, Luis and
              Smirnov, Ilya},
     TITLE = {The defect of the {F}-pure threshold},
   JOURNAL = {Adv. Math.},
  FJOURNAL = {Advances in Mathematics},
    VOLUME = {488},
      YEAR = {2026},
     PAGES = {Paper No. 110792, 50},
      ISSN = {0001-8708,1090-2082},
   MRCLASS = {13A35 (13B10 13C15 13D02 13N05)},
  MRNUMBER = {5018896},
       DOI = {10.1016/j.aim.2026.110792},
       URL = {https://doi.org/10.1016/j.aim.2026.110792},
}

@article {canton_fpt_2016,
    AUTHOR = {Canton, Eric and Hern\'andez, Daniel J. and Schwede, Karl and
              Witt, Emily E.},
     TITLE = {On the behavior of singularities at the {$F$}-pure threshold},
   JOURNAL = {Illinois J. Math.},
  FJOURNAL = {Illinois Journal of Mathematics},
    VOLUME = {60},
      YEAR = {2016},
    NUMBER = {3-4},
     PAGES = {669--685},
      ISSN = {0019-2082,1945-6581},
   MRCLASS = {13A35 (14B05 14J17)},
  MRNUMBER = {3705442},
MRREVIEWER = {Alberto\ F.\ Boix},
       URL = {https://projecteuclid.org/euclid.ijm/1506067286},
}

@article {takagi_inversion_2004,
    AUTHOR = {Takagi, Shunsuke},
     TITLE = {F-singularities of pairs and inversion of adjunction of
              arbitrary codimension},
   JOURNAL = {Invent. Math.},
  FJOURNAL = {Inventiones Mathematicae},
    VOLUME = {157},
      YEAR = {2004},
    NUMBER = {1},
     PAGES = {123--146},
      ISSN = {0020-9910,1432-1297},
   MRCLASS = {14E30 (13A35 14N30)},
  MRNUMBER = {2135186},
MRREVIEWER = {Karen\ E.\ Smith},
       DOI = {10.1007/s00222-003-0350-3},
       URL = {https://doi.org/10.1007/s00222-003-0350-3},
}

@article {de-stefani_graded_2018,
    AUTHOR = {De Stefani, Alessandro and N\'u\~nez-Betancourt, Luis},
     TITLE = {{$F$}-thresholds of graded rings},
   JOURNAL = {Nagoya Math. J.},
  FJOURNAL = {Nagoya Mathematical Journal},
    VOLUME = {229},
      YEAR = {2018},
     PAGES = {141--168},
      ISSN = {0027-7630,2152-6842},
   MRCLASS = {13A35 (13H10 14B05)},
  MRNUMBER = {3778235},
MRREVIEWER = {Thomas\ Polstra},
       DOI = {10.1017/nmj.2016.65},
       URL = {https://doi.org/10.1017/nmj.2016.65},
}

@article {bhatt_f_2015,
    AUTHOR = {Bhatt, Bhargav and Singh, Anurag K.},
     TITLE = {The {$F$}-pure threshold of a {C}alabi-{Y}au hypersurface},
   JOURNAL = {Math. Ann.},
  FJOURNAL = {Mathematische Annalen},
    VOLUME = {362},
      YEAR = {2015},
    NUMBER = {1-2},
     PAGES = {551--567},
      ISSN = {0025-5831,1432-1807},
   MRCLASS = {13A35 (13D45 14B07 14H52)},
  MRNUMBER = {3343889},
MRREVIEWER = {Linquan\ Ma},
       DOI = {10.1007/s00208-014-1129-0},
       URL = {https://doi.org/10.1007/s00208-014-1129-0},
}

@article{schwede_centers_2010,
    AUTHOR = {Schwede, Karl},
     TITLE = {Centers of {$F$}-purity},
   JOURNAL = {Math. Z.},
  FJOURNAL = {Mathematische Zeitschrift},
    VOLUME = {265},
      YEAR = {2010},
    NUMBER = {3},
     PAGES = {687--714},
      ISSN = {0025-5874,1432-1823},
   MRCLASS = {13A35 (14B05 14F18)},
  MRNUMBER = {2644316},
MRREVIEWER = {Doru\ \c Stef\u anescu},
       DOI = {10.1007/s00209-009-0536-5},
       URL = {https://doi.org/10.1007/s00209-009-0536-5},
}

@article {davenport_principle_1951,
    AUTHOR = {Davenport, Harold},
     TITLE = {On a principle of {L}ipschitz},
   JOURNAL = {J. London Math. Soc.},
  FJOURNAL = {The Journal of the London Mathematical Society},
    VOLUME = {26},
      YEAR = {1951},
     PAGES = {179--183},
      ISSN = {0024-6107,1469-7750},
   MRCLASS = {10.0X},
  MRNUMBER = {43821},
MRREVIEWER = {W.\ H.\ Mills},
       DOI = {10.1112/jlms/s1-26.3.179},
       URL = {https://doi.org/10.1112/jlms/s1-26.3.179},
}

@article {schwede_sharp_2008,
    AUTHOR = {Schwede, Karl},
     TITLE = {Generalized test ideals, sharp {$F$}-purity, and sharp test
              elements},
   JOURNAL = {Math. Res. Lett.},
  FJOURNAL = {Mathematical Research Letters},
    VOLUME = {15},
      YEAR = {2008},
    NUMBER = {6},
     PAGES = {1251--1261},
      ISSN = {1073-2780},
   MRCLASS = {13A35 (14B05)},
  MRNUMBER = {2470398},
       DOI = {10.4310/MRL.2008.v15.n6.a14},
       URL = {https://doi.org/10.4310/MRL.2008.v15.n6.a14},
}

@book{huneke_integral,
    AUTHOR = {Huneke, Craig and Swanson, Irena},
     TITLE = {Integral closure of ideals, rings, and modules},
    SERIES = {London Mathematical Society Lecture Note Series},
    VOLUME = {336},
 PUBLISHER = {Cambridge University Press, Cambridge},
      YEAR = {2006},
     PAGES = {xiv+431},
      ISBN = {978-0-521-68860-4; 0-521-68860-4},
   MRCLASS = {13B22 (13A18 13A30 13A35 13H15 14A05)},
  MRNUMBER = {2266432},
MRREVIEWER = {Liam\ O'Carroll},
}

@article{myroshnychenko_grunbaums_2018,
    AUTHOR = {Myroshnychenko, S. and Stephen, M. and Zhang, N.},
     TITLE = {Gr\"unbaum's inequality for sections},
   JOURNAL = {J. Funct. Anal.},
  FJOURNAL = {Journal of Functional Analysis},
    VOLUME = {275},
      YEAR = {2018},
    NUMBER = {9},
     PAGES = {2516--2537},
      ISSN = {0022-1236,1096-0783},
   MRCLASS = {52A40 (52A20 52A38)},
  MRNUMBER = {3847478},
MRREVIEWER = {Ai-jun\ Li},
       DOI = {10.1016/j.jfa.2018.04.001},
       URL = {https://doi.org/10.1016/j.jfa.2018.04.001},
}

@article{guardo_powers_2005,
    AUTHOR = {Guardo, Elena and Van Tuyl, Adam},
     TITLE = {Powers of complete intersections: graded {B}etti numbers and
              applications},
   JOURNAL = {Illinois J. Math.},
  FJOURNAL = {Illinois Journal of Mathematics},
    VOLUME = {49},
      YEAR = {2005},
    NUMBER = {1},
     PAGES = {265--279},
      ISSN = {0019-2082,1945-6581},
   MRCLASS = {13D40 (13D02 13H10 14A15)},
  MRNUMBER = {2157379},
MRREVIEWER = {Leah\ H.\ Gold},
       URL = {http://projecteuclid.org/euclid.ijm/1258138318},
}

@article{grunbaum_partitions_1960,
    AUTHOR = {Gr\"unbaum, Branko},
     TITLE = {Partitions of mass-distributions and of convex bodies by
              hyperplanes},
   JOURNAL = {Pacific J. Math.},
  FJOURNAL = {Pacific Journal of Mathematics},
    VOLUME = {10},
      YEAR = {1960},
     PAGES = {1257--1261},
      ISSN = {0030-8730,1945-5844},
   MRCLASS = {53.90 (52.40)},
  MRNUMBER = {124818},
MRREVIEWER = {H.\ G.\ Eggleston},
       URL = {http://projecteuclid.org/euclid.pjm/1103038065},
}

@incollection{pragacz_impanga_2012,
    AUTHOR = {Musta\c t\u a, Mircea},
     TITLE = {I{MPANGA} lecture notes on log canonical thresholds},
 BOOKTITLE = {Contributions to algebraic geometry},
    SERIES = {EMS Ser. Congr. Rep.},
     PAGES = {407--442},
      NOTE = {Notes by Tomasz Szemberg},
 PUBLISHER = {Eur. Math. Soc., Z\"urich},
      YEAR = {2012},
      ISBN = {978-3-03719-114-9},
   MRCLASS = {14B05 (13A35 14E30)},
  MRNUMBER = {2976952},
MRREVIEWER = {Ali\ Sinan\ Sert\"oz},
       DOI = {10.4171/114-1/16},
       URL = {https://doi.org/10.4171/114-1/16},
}

@article{mayes_limiting_2014,
    AUTHOR = {Mayes, Sarah},
     TITLE = {The limiting shape of the generic initial system of a complete
              intersection},
   JOURNAL = {Comm. Algebra},
  FJOURNAL = {Communications in Algebra},
    VOLUME = {42},
      YEAR = {2014},
    NUMBER = {5},
     PAGES = {2299--2310},
      ISSN = {0092-7872,1532-4125},
   MRCLASS = {13F20 (13C40 14F18 52B12)},
  MRNUMBER = {3169705},
MRREVIEWER = {Seyed\ Amin\ Seyed Fakhari},
       DOI = {10.1080/00927872.2012.758271},
       URL = {https://doi.org/10.1080/00927872.2012.758271},
}

@article{de_fernex_multiplicities_2004,
    AUTHOR = {de Fernex, Tommaso and Ein, Lawrence and Musta\c t\u a,
              Mircea},
     TITLE = {Multiplicities and log canonical threshold},
   JOURNAL = {J. Algebraic Geom.},
  FJOURNAL = {Journal of Algebraic Geometry},
    VOLUME = {13},
      YEAR = {2004},
    NUMBER = {3},
     PAGES = {603--615},
      ISSN = {1056-3911,1534-7486},
   MRCLASS = {14B12 (13H15)},
  MRNUMBER = {2047683},
MRREVIEWER = {Alexandr\ V.\ Pukhlikov},
       DOI = {10.1090/S1056-3911-04-00346-7},
       URL = {https://doi.org/10.1090/S1056-3911-04-00346-7},
}

@misc{baily_homogeneous_2026,
      title={Homogeneous ideals with minimal singularity thresholds}, 
      author={Benjamin Baily},
      year={2026},
      eprint={2603.08698},
      archivePrefix={arXiv},
      primaryClass={math.AC},
      url={https://arxiv.org/abs/2603.08698}, 
}

@article{de_fernex_bounds_2003,
    AUTHOR = {de Fernex, Tommaso and Ein, Lawrence and Musta\c t\u a,
              Mircea},
     TITLE = {Bounds for log canonical thresholds with applications to
              birational rigidity},
   JOURNAL = {Math. Res. Lett.},
  FJOURNAL = {Mathematical Research Letters},
    VOLUME = {10},
      YEAR = {2003},
    NUMBER = {2-3},
     PAGES = {219--236},
      ISSN = {1073-2780},
   MRCLASS = {14J17 (14E07 14N25)},
  MRNUMBER = {1981899},
MRREVIEWER = {Jungkai\ Alfred\ Chen},
       DOI = {10.4310/MRL.2003.v10.n2.a9},
       URL = {https://doi.org/10.4310/MRL.2003.v10.n2.a9},
}

@article{hernandez_f-purity_2016,
    AUTHOR = {Hern\'andez, Daniel J.},
     TITLE = {{$F$}-purity versus log canonicity for polynomials},
   JOURNAL = {Nagoya Math. J.},
  FJOURNAL = {Nagoya Mathematical Journal},
    VOLUME = {224},
      YEAR = {2016},
    NUMBER = {1},
     PAGES = {10--36},
      ISSN = {0027-7630,2152-6842},
   MRCLASS = {13A35 (14B05)},
  MRNUMBER = {3572748},
MRREVIEWER = {Alberto\ F.\ Boix},
       DOI = {10.1017/nmj.2016.14},
       URL = {https://doi.org/10.1017/nmj.2016.14},
}

@article {hernandez-diagonal-2015,
    AUTHOR = {Hern\'andez, Daniel J.},
     TITLE = {{$F$}-invariants of diagonal hypersurfaces},
   JOURNAL = {Proc. Amer. Math. Soc.},
  FJOURNAL = {Proceedings of the American Mathematical Society},
    VOLUME = {143},
      YEAR = {2015},
    NUMBER = {1},
     PAGES = {87--104},
      ISSN = {0002-9939,1088-6826},
   MRCLASS = {13A35 (14F18)},
  MRNUMBER = {3272734},
MRREVIEWER = {Karl\ Schwede},
       DOI = {10.1090/S0002-9939-2014-12260-X},
       URL = {https://doi.org/10.1090/S0002-9939-2014-12260-X},
}

@article {hara_on_a_2004,
    AUTHOR = {Hara, Nobuo and Takagi, Shunsuke},
     TITLE = {On a generalization of test ideals},
   JOURNAL = {Nagoya Math. J.},
  FJOURNAL = {Nagoya Mathematical Journal},
    VOLUME = {175},
      YEAR = {2004},
     PAGES = {59--74},
      ISSN = {0027-7630,2152-6842},
   MRCLASS = {13A35},
  MRNUMBER = {2085311},
MRREVIEWER = {Ana\ Bravo},
       DOI = {10.1017/S0027763000008904},
       URL = {https://doi.org/10.1017/S0027763000008904},
}

@article {hara_generalization_2003,
    AUTHOR = {Hara, Nobuo and Yoshida, Ken-Ichi},
     TITLE = {A generalization of tight closure and multiplier ideals},
   JOURNAL = {Trans. Amer. Math. Soc.},
  FJOURNAL = {Transactions of the American Mathematical Society},
    VOLUME = {355},
      YEAR = {2003},
    NUMBER = {8},
     PAGES = {3143--3174},
      ISSN = {0002-9947,1088-6850},
   MRCLASS = {13A35},
  MRNUMBER = {1974679},
MRREVIEWER = {Irena\ Swanson},
       DOI = {10.1090/S0002-9947-03-03285-9},
       URL = {https://doi.org/10.1090/S0002-9947-03-03285-9},
}

@article{takagi_f-pure_2004,
    AUTHOR = {Takagi, Shunsuke and Watanabe, Kei-ichi},
     TITLE = {On {F}-pure thresholds},
   JOURNAL = {J. Algebra},
  FJOURNAL = {Journal of Algebra},
    VOLUME = {282},
      YEAR = {2004},
    NUMBER = {1},
     PAGES = {278--297},
      ISSN = {0021-8693,1090-266X},
   MRCLASS = {13A35},
  MRNUMBER = {2097584},
MRREVIEWER = {Oleg\ N.\ Popov},
       DOI = {10.1016/j.jalgebra.2004.07.011},
       URL = {https://doi.org/10.1016/j.jalgebra.2004.07.011},
}

\end{document}